\documentclass[11pt]{article}
\usepackage[utf8]{inputenc}
\usepackage[english]{babel}
\usepackage{blindtext}
\usepackage{tcolorbox}
\usepackage{graphicx}
\usepackage{import}
\usepackage{amsmath}
\usepackage{amssymb}    
\usepackage{bbm}
\usepackage{esint}
\usepackage{cancel}
\usepackage{proof}
\usepackage[margin=1.25in]{geometry}
\usepackage{amsthm}
\usepackage{todonotes}
\usepackage{asymptote}
\usepackage{xcolor}
\usepackage{mathrsfs}
\usepackage[colorlinks=true]{hyperref}

\definecolor{heavygreen}{RGB}{65,190,0}

\newtheorem{theorem}{Theorem}[section]

\newtheorem{lemma}[theorem]{Lemma}

\newcommand{\Z}{\mathbb{Z}}

\theoremstyle{definition}
\newtheorem{definition}[theorem]{Definition}
\newtheorem{remark}[theorem]{Remark}

\begin{document}
\title{Patterns in the Lattice Homology of Seifert Homology Spheres}
\author{Karthik Seetharaman, William Yue, and Isaac Zhu}
\date{October 22, 2021}

\maketitle

\begin{abstract}
In this paper, we study various homology cobordism invariants for Seifert fibered integral homology 3-spheres derived from Heegaard Floer homology. Our main tool is lattice homology, a combinatorial invariant defined by Ozsv\'ath-Szab\'o and N\'emethi. We reprove the fact that the $d$-invariants of Seifert homology spheres $\Sigma(a_1,a_2,\dots,a_n)$ and $\Sigma(a_1,a_2,\dots,a_n+a_1a_2\cdots a_{n-1})$ are the same using an explicit understanding of the behavior of the numerical semigroup minimally generated by $a_1a_2\cdots a_n/a_i$ for $i\in[1,n]$. We also study the maximal monotone subroots of the lattice homologies, another homology cobordism invariant introduced by Dai and Manolescu. We show that the maximal monotone subroots of the lattice homologies of Seifert homology spheres $\Sigma(a_1,a_2,\dots,a_n)$ and $\Sigma(a_1,a_2,\dots,a_n+2a_1a_2\cdots a_{n-1})$ are the same.

\end{abstract}


\section{Introduction}

The homology cobordism group $\Theta_{\Z}^3$ is a well-studied object in low-dimensional topology, and there have been many attempts in the last few decades to understand its structure. For example, it is known that $\Theta_{\Z}^3$ has a $\mathbb{Z}$-summand (see \cite{froyshov2002equivariant}) and contains $\mathbb{Z}^\infty$ as a subgroup (see \cite{fintushel1990instanton} and \cite{furuta1990homology}); it was recently proven that it also admits a $\mathbb{Z}^\infty$-summand \cite{dai2018infinite}. 

A common tool for studying the homology cobordism group is \textit{Heegaard Floer homology}, an invariant of 3-manifolds defined by Ozsv\'ath and Szab\'o in \cite{ozsvath2004holomorphic}. Heegaard Floer homology is very successful for studying the homology cobordism group, but in general it is very difficult to compute.

However, for certain classes of manifolds, the Heegaard Floer homology is isomorphic to another combinatorially-defined invariant known as \textit{lattice homology}, which is easier to understand and compute. One such class of manifolds for which this is true is \textit{Seifert fibered integral homology spheres}, which are an important class of examples in the study of the homology cobordism group (see \cite{cochran2014homology} and \cite{fintushel1990instanton}). 

In this paper, we study the lattice homologies of Seifert fibered integral homology spheres and related homology cobordism invariants. For these manifolds, work of Can and Karakurt in \cite{can2014calculating} allows us to compute lattice homology using the $\tau$-sequence (refer to Subsection \ref{tau sequence intro}), which is derived by the numerical semigroup minimally generated by $a_1a_2\cdots a_n/a_i$ for $i\in[1,n]$. Though this reformulation is easier to compute, it is still complicated and far from closed-form. There is plenty of interest in computing the lattice homologies of Brieskorn spheres, Seifert fibered integral homology spheres with three fibers, such as in \cite{durusoy2004heegaard}, \cite{saveliev1999floer}, and \cite{tweedy2013heegaard}. 

Specifically, we study the relationship between lattice homologies of families of Seifert fibered integral homology spheres of the form 
\[\Sigma(a_1,a_2,\dots,a_{n-1},a_n+k\alpha),\qquad\text{where}\qquad\alpha:=a_1a_2\cdots a_{n-1}\]
and $k\in \mathbb{Z}$. In particular, we prove periodicity results about homology cobordism invariants within these families.

In the first part of our paper, we focus on the $d$-invariants of these spheres. The $d$-invariant is a numerical invariant derived from Heegaard Floer homology (and thus in our case the lattice homology), of a 3-manifold, and specifies a surjective homomorphism from $\Theta_{\Z}^3$ to $2\mathbb{Z}$ as stated in \cite{hendricks2018connected}. 


We prove the following theorem about $d$-invariants:

\begin{theorem}\label{dinvariant}
The $d$-invariants of the two Seifert fibered integral homology spheres 
\[\Sigma(a_1,a_2,\dots,a_{n-1},a_n) \qquad\text{and}\qquad
\Sigma(a_1,a_2,\dots,a_{n-1},a_n+\alpha)\]
are equal for pairwise relatively prime positive integers $a_1,a_2,\ldots,a_n \geq 2$. Recall that $\alpha:=a_1a_2\cdots a_{n-1}$.
\end{theorem}

\begin{remark}
The result on $d$-invariants was proven as Proposition 4.1 of \cite{lidman2018note} by interpreting the $+\alpha$ term as surgery on a singular fiber. Our proof is a consequence of understanding the relation between the lattice homologies of the two spaces in question. This involves a more explicit understanding of the $\tau$-sequence and related $\Delta$-function. Although this method is more complicated, it proves helpful in our later results about the maximal monotone subroot (refer to Theorem \ref{monotonesubroot}).
\end{remark}

The second part of the paper is dedicated to the study of the \textit{maximal monotone subroot} of Seifert homology spheres, which was introduced in \cite{dai2019involutive} recently. The maximal monotone subroot is another homology cobordism invariant that is defined for certain plumbed 3-manifolds, including all Seifert homology spheres, which can be derived from their lattice homology. As there is plenty of interest in understanding the full lattice homology, it is natural to try to understand the maximal monotone subroot as well. 

In this paper, we prove the following theorem about the maximal monotone subroots of the lattice homologies of Seifert fibered integral homology spheres.

\begin{theorem}\label{monotonesubroot}
The maximal monotone subroots of the lattice homologies of the two Seifert fibered integral homology spheres 
\[\Sigma(a_1,a_2,\ldots,a_n)\qquad\text{and}\qquad\Sigma(a_1,a_2,\ldots,a_n+2\alpha)\]
are the same for pairwise relatively prime positive integers $a_1,a_2,\ldots,a_n \geq 2$. Recall that $\alpha:=a_1a_2\cdots a_{n-1}$.
\end{theorem}

\begin{remark}
Certain cases of this theorem follow from more general work in \cite{hendricks2020surgery}, which establishes a surgery formula for involutive Heegaard Floer homology. A consequence of Proposition 22.9 in that paper is that the maximal monotone subroots of the lattice homologies of $\Sigma(p,q,pqn\pm 1)$ and $\Sigma(p,q,pq(n+2)\pm 1)$ are the same. In Theorem \ref{monotonesubroot}, we generalize this result both to an arbitrary number of fibers and for $a_n$ to be an arbitrary residue modulo $\alpha$ (it is not restricted to just $\pm 1\pmod{\alpha}$) using our understanding of the lattice homology through the $\tau$-sequence, $\Delta$-function, and the numerical semigroup minimally generated by $a_1a_2\cdots a_n/a_i$ for $i\in[1,n]$.
\end{remark}

\begin{remark}\label{importance d-invariant and maximal monoton subroot}
Note that if an integral homology sphere is homology cobordant to $S^3$, then it bounds an integral homology ball. Since $S^3$ has a $d$-invariant of $0$ and a trivial maximal monotone subroot (that is, just a single upwards-pointing infinite stem), 
in order for a Seifert fibered integral homology sphere to bound an integral homology 4-ball, its $d$-invariant must be $0$ and its maximal monotone subroot must be trivial. 
\end{remark}

\begin{remark}
It has been shown in Theorem 2.2 of \cite{manolescu2018homology} 
that the classes $[\Sigma(2,3,6k-1)]$ are linearly independent in $\Theta_{\mathbb Z}^3$ using Yang--Mills Theory. This is in constrast to the $d$-invariant, which we've shown in Theorem \ref{dinvariant} to be unable to distinguish any of these classes, and the maximal monotone subroot, which we've shown in Theorem \ref{monotonesubroot} to at most be able to distinguish the parity of $k$. Recall that Proposition 22.9 in \cite{hendricks2020surgery} demonstrated this already for this particular class of 3-fibered Seifert fibered integral homology spheres where $a_n\equiv\pm 1\pmod{\alpha}$.
\end{remark}

\paragraph{Organization.} In Section~\ref{prelims}, we define Seifert fibered integral homology spheres, lattice homology, and review a numerical method of  computing the lattice homology of these 3-manifolds. In Section~\ref{properties of tau}, we analyze and prove several properties of the $\tau$-sequence, the related $\Delta$-function, and the numerical semigroup minimally generated by $a_1a_2\cdots a_n/a_i$ for $i\in[1,n]$. We also introduce a useful pictorial representation of the $\Delta$-function. Finally, in Section~\ref{results d invariants}, we prove Theorem \ref{dinvariant}, and in Section~\ref{results maximal monotone subroots} we prove Theorem \ref{monotonesubroot}.

\paragraph{Acknowledgements.} We would like to thank our mentor Dr. Irving Dai for his guidance throughout the project, as well as the MIT PRIMES program under which this research was conducted.

\section{Preliminaries}\label{prelims}

In this section, we recall the definitions of Seifert fibered integral homology spheres (which we will call \textit{Seifert homology spheres} from now on) and lattice homology and then review a method of numerically computing the lattice homology of Seifert homology spheres.

\subsection{Seifert fibered integral homology spheres}\label{sf integral homology spheres}

\begin{definition}
    Let $a_1,a_2,\ldots,a_n$ be pairwise coprime positive integers. Solve the equation 
    \begin{equation}\label{homologyspherecondition}
    \sum_{i=1}^n \frac{b_i}{a_i} \left(\prod_{j=1}^n a_j \right) = -1 - e_0a_1a_2 \cdots a_n
    \end{equation}
    for $(e_0,b_1,b_2,\ldots,b_n)$, where we restrict $1\le b_i<a_i$ for all $i\in[1,n]$. Taking this equation modulo $a_i$ gives
    \begin{equation}\label{computing b}
    \frac{a_1a_2\cdots a_n b_i}{a_i}\equiv -1\pmod{a_i},
    \end{equation}
    so there is a unique solution for each $b_i$. Note that $e_0<0$. Then the \textit{Seifert homology sphere} $\Sigma(a_1,a_2,\ldots,a_n)$ is defined as the generalized Seifert manifold 
    \[M(e_0,(a_1,b_1),(a_2,b_2),\ldots, (a_n,b_n))\]
    over $S^2$, with surgery diagram shown in Figure~\ref{fig:surgery diagram}. Note that equation (\ref{homologyspherecondition}) ensures that this manifold is an integral homology sphere.
\end{definition}


\begin{figure}
    \centering
    \includegraphics[width=7cm]{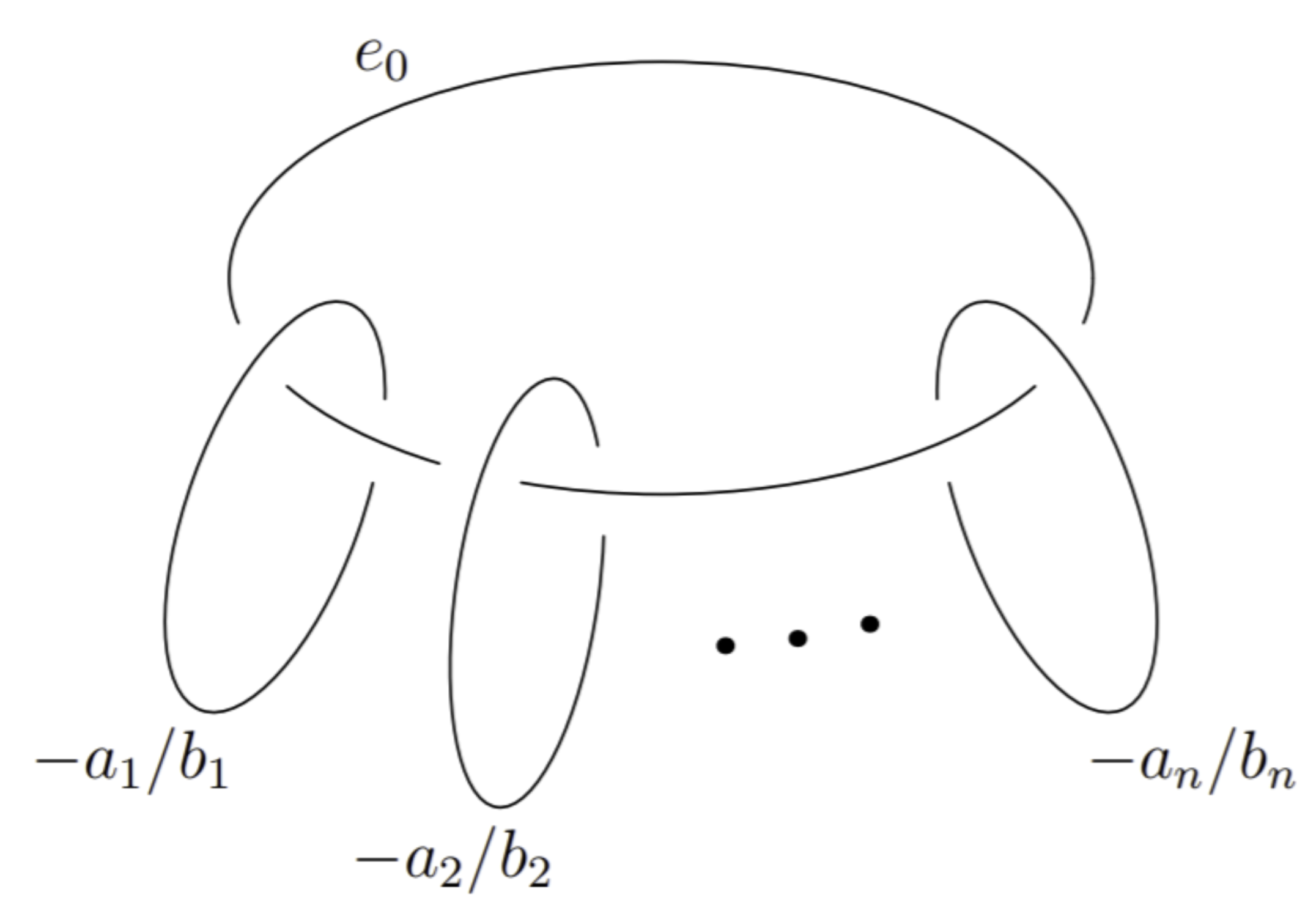}
    \caption{Rational surgery diagram for the Seifert homology sphere $\Sigma(a_1,a_2,\dots,a_n)$}
    \label{fig:surgery diagram}
\end{figure}

\subsection{Lattice homology}

This paper investigates the \textit{lattice homologies} of these Seifert homology spheres. In this section, we define the lattice homology, which is an invariant defined for plumbed 3-manifolds together with a chosen equivalence class of characteristic vector (refer to Definitions \ref{charvector} and \ref{charvectorequivalence}). Though the exact details of the definition of lattice homology will be unnecessary for understanding the rest of the paper, we include them for completeness. 

\begin{definition}
A \textit{plumbed 3-manifold} is a manifold with a surgery diagram consisting of integral surgeries on unknots linked together in a tree.
\end{definition}

To define lattice homology, we use the \textit{plumbing graph} of a plumbed 3-manifold $Y$, which is a decorated tree that represents the integral surgery diagram of $Y$, created by replacing each unknot with a vertex and connecting an edge between vertices if their respective unknots are linked. For the particular case of Seifert homology spheres $\Sigma(a_1,a_2,\dots,a_n)$, we can obtain this by replacing each rational surgery in Figure~\ref{fig:surgery diagram} with a chain of unknots with coefficients determined by partial fraction decomposition; the integral surgery diagram and plumbing graph of $\Sigma(a_1,a_2,\dots,a_n)$ is shown in Figure~\ref{fig:integral surgery diagram and plumbing graph}. 


The definition of lattice homology requires that the manifold admits a negative-definite plumbing graph, and for the rest of this section, we will assume that all plumbing graphs are negative-definite. 

\begin{figure}
    \centering
    \includegraphics[width=14cm]{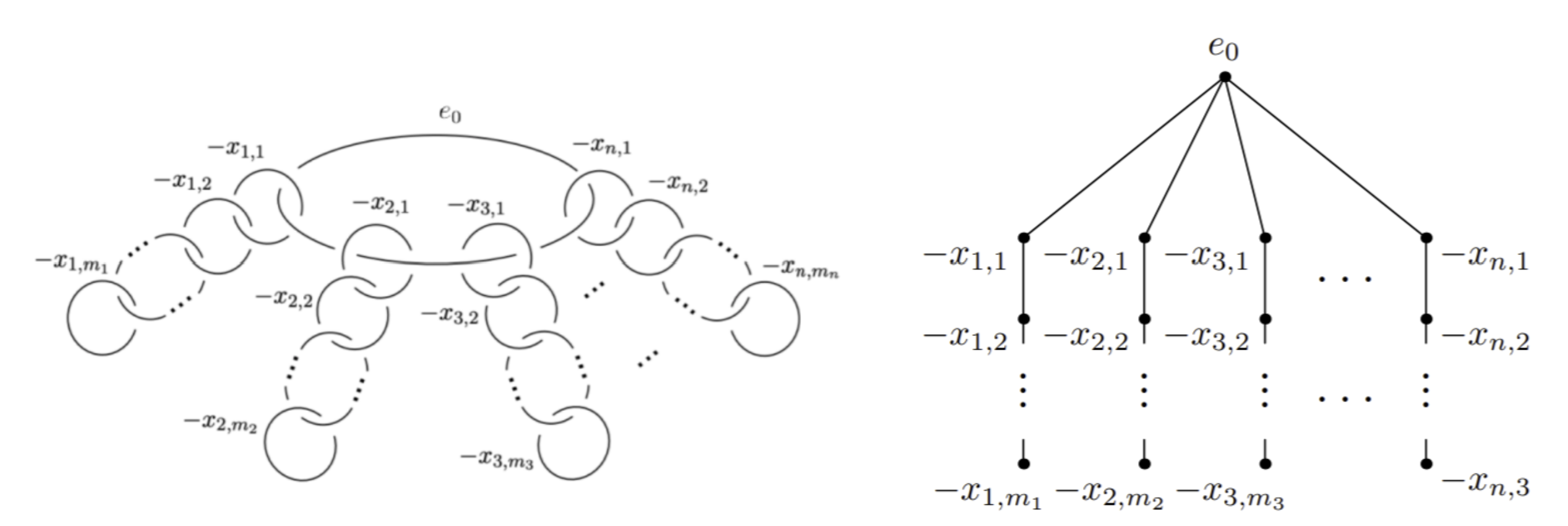}
    \caption{Integral surgery diagram (on left) and associated plumbing graph (on right) of Seifert homology sphere $\Sigma(a_1,a_2,\cdots,a_n)$, where the partial fraction decomposition of $\frac{a_i}{b_i}$ is $[x_{i,1},x_{i,2},\dots,x_{i,m_i}]$ for each $i\in [1,n]$.}
    \label{fig:integral surgery diagram and plumbing graph}
\end{figure}

\begin{definition}
The \textit{intersection form} $(-,-)$ of a given plumbing graph $\Gamma$ of a plumbed 3-manifold $Y$ is defined on the lattice $L_{\Gamma}=\text{Span}_{\mathbb{Z}}(v_1,v_2,\ldots,v_n)$ formally spanned by the vertices $v_1,v_2,\ldots, v_n$ of $\Gamma$. It is given by
\[(v_i,v_j)= \begin{cases} 
      0 & i \neq j \text{ and } v_i,v_j \text{ not adjacent} \\
      1 & i \neq j \text{ and } v_i,v_j \text{ adjacent} \\
      \text{decoration of } v_i & i=j
   \end{cases},
\]
which is then extended bilinearly to all of $L_\Gamma\otimes\mathbb{Q}$. Note that this is just the adjacency matrix of $\Gamma$, except vertices have a self-adjacency equal to their decoration. 
\end{definition}

\begin{definition}\label{charvector}
Let $k$ be an element of the rational lattice $L_{\Gamma} \otimes \mathbb{Q}$. We say $k$ is a \textit{characteristic vector} if 
\[(k,x) \equiv (x,x) \pmod{2}\]
for all $x \in L_{\Gamma}$. The set of characteristic vectors is denoted as $\text{Char}_{\Gamma}$.
\end{definition}

\begin{definition}\label{charvectorequivalence}
Given any characteristic vector $k\in\text{Char}_{\Gamma}$ and element $y\in L_\Gamma$, the vector $k+2y$ is also characteristic. This action of $L_\Gamma$ partitions the set of characteristic vectors into equivalence classes, and we denote the equivalence class of $k$ by $[k]$.
\end{definition}

We can now define lattice homology as in~\cite{nemethi2008lattice}, which is an invariant of a plumbed 3-manifold along with a chosen equivalence class of characteristic vector. We will do so using sublevel sets. Let $\Gamma$ be any plumbing graph and fix $k \in \text{Char}_{\Gamma}$. We define a weight function 
\[w(x) = \frac{(x,x)+(k,x)}{2}\]
and extend it to $d$-dimensional cubes of side-length one by setting 
\[w(x_d) = \min_{x \text{ a vertex of } x_d} w(x),\]
where $x_d$ is an $d$-dimensional cube. For any $n$, the \textit{sublevel set} $S_n$ is defined as 
\[S_n = \bigcup_{w(y) \geq n} y,\]
where $y$ is any $d$-dimensional cube with side length one. We then have the following:  

\begin{definition}
Fix a plumbing graph $\Gamma$ for a plumbed 3-manifold $Y$, and for each $n \in \mathbb{Z}$, draw a vertex for each connected component in the sublevel set $S_n$ at height $-2n$ on the page. Note that $S_{n+1} \subseteq S_n$ for all $n \in \mathbb{Z}$, so for each $n$ define the map $T_n : S_{n+1} \rightarrow S_n$ that sends each connected component in $S_{n+1}$ to the one it is contained in inside $S_n$. Record the results of this map by drawing lines between corresponding vertices in the diagram. This gives a \textit{graded root}, as shown in Figure~\ref{fig:lattice homology example}. The lattice homology is this graded root with all heights shifted by the quantity $-\frac{(k,k) + |\Gamma|}{4}$; note that the final height of each vertex is called its \textit{grading}.
\end{definition}

Note that since the sublevel sets consist of the $d$-dimensional cubes with side length one contained within some expanding ellipsoid, there will eventually be only one connected component for all sufficiently negative $n$. This corresponds to the infinite tower on top of the lattice homology. On the other hand, when $n$ is larger than the maximum of the weight function on the intersection form, $S_n$ contains no connected components, which corresponds to the lattice homology ending when it is sufficiently low.

\begin{figure}
    \centering
    \includegraphics[width=10cm]{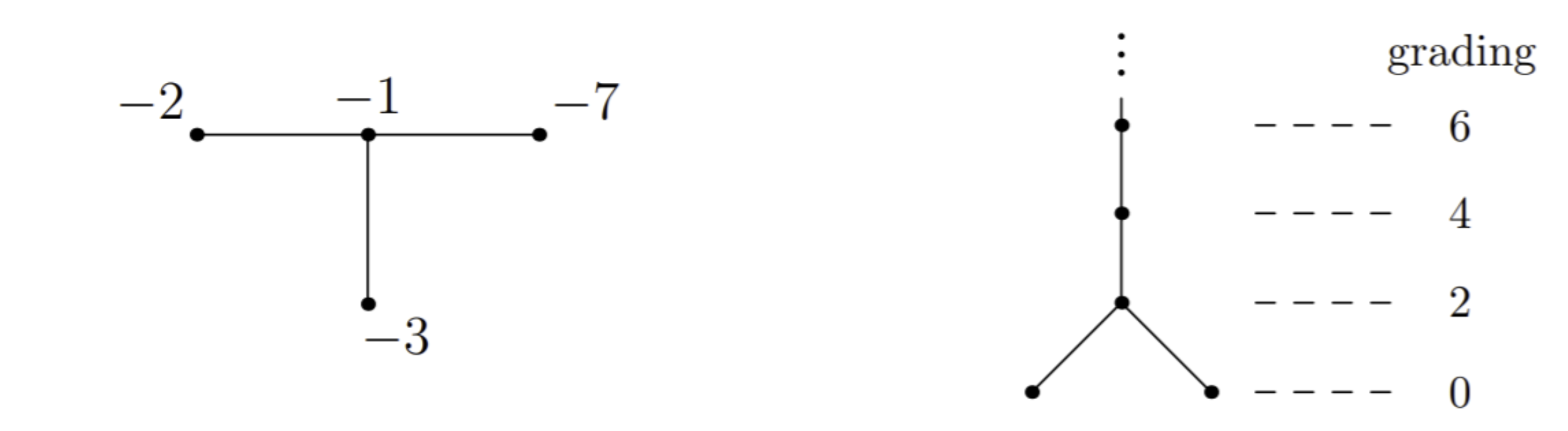}
    \caption{The plumbing graph of $\Sigma(2,3,7)$ on the left with its graded root lattice homology on the right.}
    \label{fig:lattice homology example}
\end{figure}

\begin{remark}
Remarkably, it can be shown that, up to equivalence classes of characteristic vectors, the lattice homology does not depend on the particular plumbing diagram chosen for a plumbed 3-manifold; that is, it is an invariant of plumbed 3-manifolds themselves. Note, however, that while lattice homology does not depend on the specific characteristic vector in a chosen equivalence class, it does vary if a different equivalence class of characteristic vectors is chosen.
\end{remark}

As stated in \cite{ozsvath2003floer}, every Seifert homology sphere has a negative-definite plumbing graph, and thus a lattice homology. Furthermore, because they are integral homology spheres, it turns out that these manifolds only admit one equivalence class of characteristic vector anyway, so each only has one lattice homology. It is also known that, for certain plumbed 3-manifolds with negative-definite plumbing (such as Seifert homology spheres), the lattice homology is isomorphic to the Heegaard-Floer homology \cite{ozsvath2003floer}, a long-studied and important invariant whose definition is outside the scope of this paper.

We can now define the $d$-invariant of a 3-manifold, as in~\cite{hom2019heegaard}.

\begin{definition}
The $d$\textit{-invariant} of a 3-manifold is $-1$ times the grading of the lowest vertex of its lattice homology. 
\end{definition}

\begin{remark}
The factor of $-1$ is simply for convention reasons.
\end{remark}

\subsection{Constructing the lattice homology of Seifert homology spheres using the $\tau$-sequence}\label{tau sequence intro}

As it turns out, for Seifert homology spheres, the lattice homology and $d$-invariant can be understood through the behavior of a particular sequence known as the $\tau$-sequence.

\begin{definition}\label{tau}
Consider a Seifert homology sphere $\Sigma(a_1,a_2,\ldots,a_n)$. The $\tau$\textit{-sequence} is defined by the recurrence 
\[\tau(x+1)=\tau(x)+1+|e_0|x - \sum_{i=1}^n \left\lceil{\frac{xb_i}{a_i}}\right\rceil,\]
where, as before, 
\[\sum_{i=1}^n \frac{b_i}{a_i} \left(\prod_{j=1}^n a_j \right) = -1 - e_0a_1a_2 \ldots a_n.\]
\end{definition}

\begin{definition}
We define the difference term in the above recurrence as the $\Delta$\textit{-function}
\[\Delta(x)=1+|e_0|x-\sum_{i=1}^n\left\lceil\frac{xb_i}{a_i}\right\rceil.\]
Therefore, 
\[\tau(x)=\sum_{i=0}^{x-1}\Delta(i).\]
\end{definition}

Now, we review a method of computing the lattice homology of Seifert homology spheres using this $\tau$-sequence. 

\begin{definition}
We say that $\tau(M)$ is a \textit{local maximum} of $\tau$ if there exists integers $\alpha,\beta$ with $\alpha<M<\beta$ such that $\tau(\alpha)<\tau(M)>\tau(\beta)$, and $\tau$ is monotone nondecreasing on the interval $[\alpha,M]$ and monotone nonincreasing on the interval $[M,\beta]$. \textit{Local minimum} values $\tau(m)$ are defined analogously. Together, these are called the \textit{local extrema} of $\tau$.
\end{definition}

\begin{definition}
Consider the sequence of all local extrema of the $\tau$-sequence. However, sometimes the $\tau$-sequence remains constant at a local extrema for multiple consecutive inputs. We choose to count these consecutive repeated local extrema as a single value. For example, the sequence $[0,1,1,1,-1,-1,1,1,1,0]$ of all local extrema (including consecutive repeated ones) of $\tau$ is collapsed into $[0,1,-1,1,0]$. This collapsed sequence, denoted $\tau_{\text{ex}}$, and is called the $\tau$\textit{-extrema sequence}.
\end{definition}

We can associate $\tau_{\text{ex}}$ with a graded root, which, after a grading shift, gives the lattice homology of the Seifert homology sphere. For any graded root, we can associated it to a sequence through the following procedure: given any graded root, we consider a path that wraps around the tree, starting to the left of the infinite stem and ending to its right. As an example, in Figure \ref{fig:graded root path}, the path is drawn in blue.

\begin{figure}
    \centering
    \includegraphics[width=11cm]{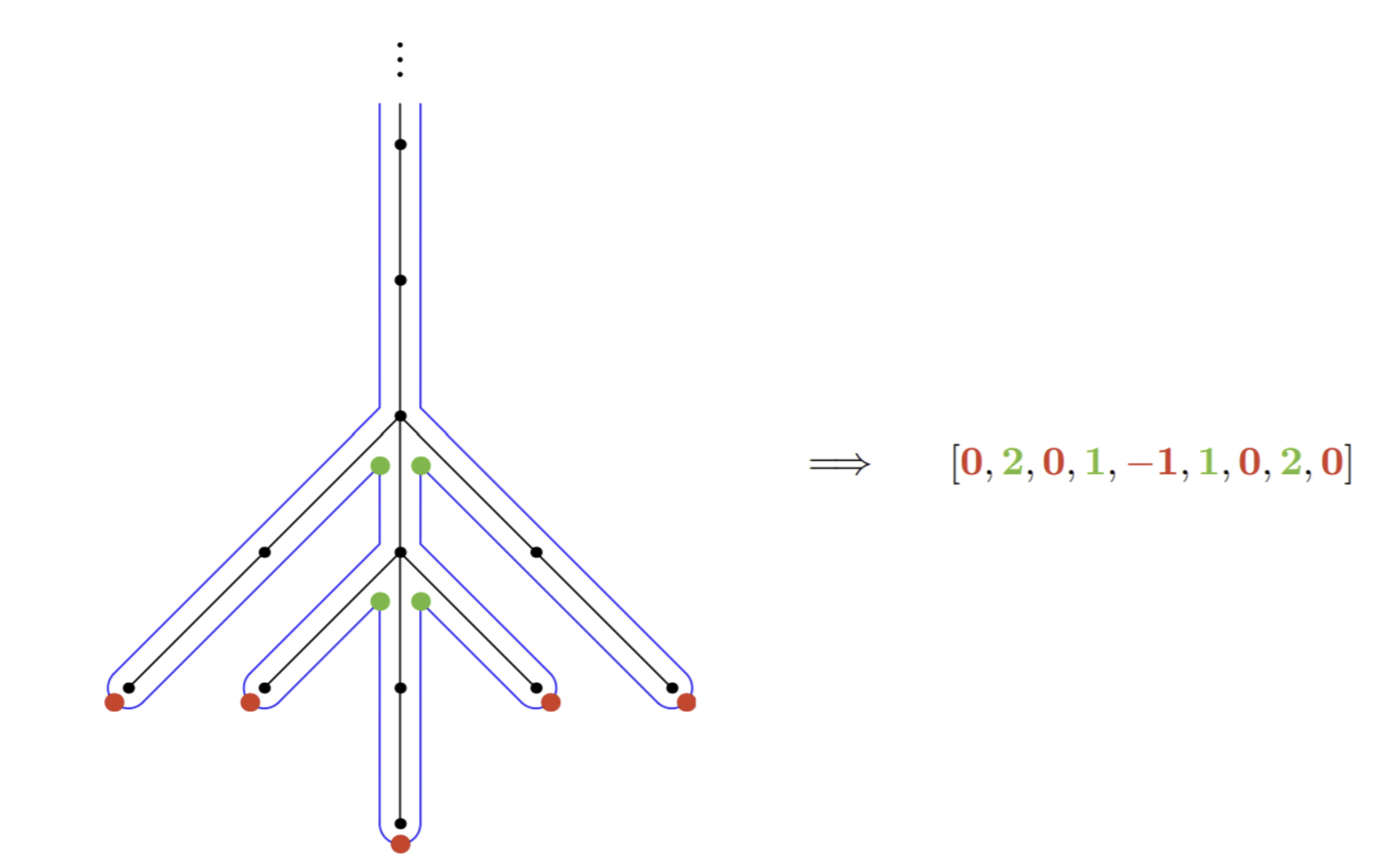}
    \caption{Example of extracting sequence from graded root, assuming the grading of the leftmost vertex is $0$.}
    \label{fig:graded root path}
\end{figure}

Then, we mark every local minimum of this blue path with a red point and every local maximum of this blue path with a green point; more formally, red points are where the blue path changes from moving downwards to moving upwards, and green points are where the blue path changes from moving upwards to moving downwards. Recording the gradings of the vertices of the graded root where these local extrema occur in order along the blue path gives the sequence associated to this graded root. In the example in Figure \ref{fig:graded root path}, we get the sequence $[0,2,0,1,-1,1,0,2,0]$. In general, this procedure gives some sequence
\[[m_1,M_1,m_2,M_2,\dots,m_{n-1},M_{n-1},m_n]\]
of numbers such that $M_i>\max\{m_i,m_{i+1}\}$ for $i=1,2,\dots,n-1$. Conversely, any such sequence $s$ also uniquely corresponds to a graded root, which we denote by $R_s$. 

Since $2\tau_{\text{ex}}$ is a sequence of alternating local minima and maxima, it has an associated graded root $R_{2\tau_{\text{ex}}}$. It turns out that this graded root, after a global grading shift, matches the lattice homology of the Seifert homology sphere. Note that the factor of 2 ensures that all gradings are the same parity, as in the lattice homology.

\begin{theorem}[\cite{nemethi2005ozsvath}]
The lattice homology of the Seifert homology sphere $\Sigma(a_1,a_2,\dots,a_n)$ is isomorphic to $R_{2\tau_{\text{ex}}}$ after applying a global grading shift $-\frac{K^2+|\Gamma|}{4}$, where $K$ is the canonical cohomology class, which can be viewed as a specially selected characteristic vector.
\end{theorem}

\begin{remark}
It turns out that the graded root $R_{2\tau_{\text{ex}}}$ is symmetric (which is required since the lattice homology is necessarily symmetric by definition). This is true because of Property 2 of Theorem~\ref{deltadesc}.
\end{remark}

\section{Properties of the $\tau$-Sequence and $\Delta$-Function}\label{properties of tau}

In this section we will develop a pictorial representation of the $\Delta$-function, by constructing a table of the integers with $\alpha:=a_1a_2 \cdots a_{n-1}$ columns. This picture will be critical to our later analysis. Figure \ref{fig:semigroup picture} depicts the function for $(a_1,a_2,a_3)=(3,7,29)$, where the width of the grid is $3 \cdot 7 = 21$.

\begin{figure}[h]
    \centering
    \includegraphics[width=10cm]{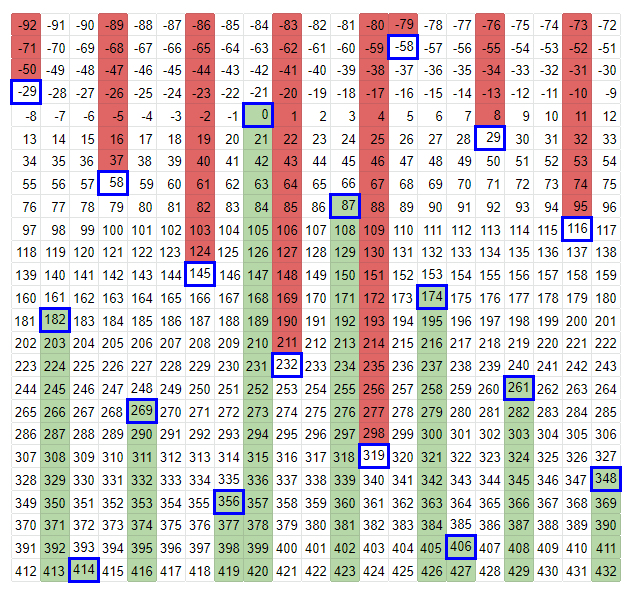}
    \caption{This diagram represents the $\Delta$-function pictorally, and it displays information for $(p,q,r)=(3,7,29)$ with some arbitrarily chosen bounds. As we will soon see, the blue bordered boxes are multiples of $a_n$, the green boxes have $\Delta(x)=0$, the red boxes have $\Delta(x)=-1$, and the white boxes have $\Delta(x)=0$. As some random examples, $\Delta(-44)=-1$, $\Delta(145)=1$, and $\Delta(339)=1$.}
    \label{fig:semigroup picture}
\end{figure}

Let $\Delta: \mathbb{Z} \rightarrow \mathbb{Z}$ be given by
\[\Delta(x)=1+|e_0|x-\sum_{i=1}^n\left\lceil\frac{xb_i}{a_i}\right\rceil\]
for all $x\in\mathbb{Z}$. With this consideration of domain, we have the following natural extension of Theorem 4.1 in \cite{can2014calculating}, by Can and Karakurt:
\begin{theorem}\label{deltadesc}
Let $a_1,a_2,\dots,a_n$ be pairwise relatively prime positive integers, and define 
\[\Delta(x)=1+|e_0|x-\sum_{i=1}^n\left\lceil\frac{xb_i}{a_i}\right\rceil,\]
as before. Define the constant
\[N_0=a_1a_2\cdots a_n\left((n-2)-\sum_{i=1}^n\frac{1}{a_i}\right)\in\mathbb{Z}_{>0}.\]
Then, the following properties hold: 
\begin{enumerate}
    \item $\Delta(x)\ge 0$ for all $x>N_0$.
    \item $\Delta(x)=-\Delta(N_0-x)$ for \textnormal{all} $x\in\mathbb{Z}$.
    \item For \textnormal{all} $x\in\mathbb{Z}$, one has $\Delta(x)\ge 1$ if and only if $x$ is an element of the numerical semigroup $G$ minimally generated by $a_1a_2\cdots a_n/a_i$ for $i\in[1,n]$ (defined below in Definition \ref{numsemigroup}).
    \item If $x\in G$ and $\displaystyle x=a_1a_2\cdots a_n\sum_{i=1}^n\frac{x_i}{a_i}$, then $\displaystyle\Delta(x)=1+\sum_{i=1}^n\left\lfloor\frac{x_i}{a_i} \right\rfloor$.
\end{enumerate}
\end{theorem}

\begin{definition}\label{numsemigroup}
The numerical semigroup $G$ minimally generated by $a_1a_2\cdots a_n/a_i$ for $i\in[1,n]$ is the set of all $x\in\mathbb{Z}$ that can be expressed as a nonnegative integer linear combination of $a_1a_2\cdots a_n/a_i$. That is,
\[G=\left\{x\in\mathbb{Z} : x=a_1a_2\cdots a_n\sum_{i=1}^n\frac{x_i}{a_i}\text{ for some }x_i\in\mathbb{Z}_{\ge 0}\right\}.\]
\end{definition}

\begin{remark}
Note that since $\Delta(x)\ge 0$ for all $x>N_0$, and Property 2 of Theorem~\ref{deltadesc} gives that $\Delta(x)\le 0$ for all $x<0$, the $\tau$-sequence has no local extrema outside of the interval $[0,N_0]$, so this interval is all that matters for constructing the lattice homology.
\end{remark}

Let $H$ denote the numerical semigroup minimally generated by the numbers $\alpha/a_i$ for $1 \le i \le n-1$.

\begin{lemma}\label{primitive semigroup}
If $s \in H$ but $s-\alpha \not\in H$, then $\Delta(sa_n)=1$.
\end{lemma}

\begin{proof}
Since $s \in H$, there are nonnegative integers $x_i$ such that
\[sa_n = \left( \alpha \sum_{i=1}^{n-1} \frac{x_i}{a_i} \right) a_n =a_1a_2 \cdots a_n \sum_{i=1}^{n-1}\frac{x_i}{a_i}.\]Notice that we must have $x_i<a_i$ for each $i$: otherwise, we would have $s-\alpha \in H$. So by Theorem~\ref{deltadesc} we have 
\[\Delta(sa_n) = 1+\sum_{i=1}^n\left\lfloor\frac{x_i}{a_i} \right\rfloor=1.\]
\end{proof}

We now describe the relationship between the $\Delta$s of vertically adjacent cells in the grid (note that $x$ and $x+\alpha$ are vertically adjacent):

\begin{lemma}\label{delta in columns}
We have that 
\[\Delta(x+\alpha)=\begin{cases}
\Delta(x)+1 & \text{if } a_n\mid x+\alpha\\
\Delta(x) & \text{otherwise}
\end{cases}.\]
\end{lemma}

\begin{proof}
We apply the definition of $\Delta$ to get 
\begin{align*}
\Delta&(x+\alpha)-\Delta(x) = |e_0|\alpha - \sum_{i=1}^n \left(\left\lceil{\frac{(x+\alpha)b_i}{a_i}}\right\rceil - \left\lceil{\frac{xb_i}{a_i}}\right\rceil\right)\\
&=-e_0\alpha-\sum_{i=1}^{n-1}\frac{\alpha b_i}{a_i}-\left\lceil\frac{(x+\alpha)b_n}{a_n}\right\rceil+\left\lceil\frac{xb_n}{a_n}\right\rceil.
\end{align*}
Equation (\ref{homologyspherecondition}) gives
\[e_0\alpha = -\frac{1}{a_n} - \sum_{i=1}^{n-1} \frac{\alpha b_i}{a_i} - \frac{\alpha b_n}{a_n},\]
so 
\[\Delta(x+\alpha)-\Delta(x) = \frac{1+\alpha b_n}{a_n} - \left\lceil{\frac{(x+\alpha)b_n}{a_n}}\right\rceil + \left\lceil{\frac{xb_n}{a_n}}\right\rceil.\]
Now, define the function $\psi(x):=\lceil x\rceil - x$. Plugging this in gives
\[\Delta(x+\alpha)-\Delta(x)=\frac{1}{a_n}-\psi\left(\frac{(x+\alpha)b_n}{a_n}\right)+\psi\left(\frac{xb_n}{a_n}\right).\]
By equation (\ref{computing b}), we have that $\alpha b_n\equiv -1\pmod{a_n}$. Since $\psi(x+k)=\psi(x)$ for all $k\in\mathbb{Z}$, we have
\[\Delta(x+\alpha)-\Delta(x)=\frac{1}{a_n}+\psi\left(\frac{xb_n}{a_n}\right)-\psi\left(\frac{xb_n-1}{a_n}\right).\]
This quantity equals $1$ if and only if $xb_n-1\equiv 0\pmod{a_n}$; in all other cases, it equals $0$. However, $xb_n-1\equiv 0\pmod{a_n}$ is equivalent to $x\equiv -\alpha\pmod{a_n}$. Therefore, $\Delta(x+\alpha)-\Delta(x)$ equals 1 if $a_n\mid x+\alpha$ and 0 otherwise, as desired.
\end{proof}

Recalling that $x+\alpha$ is the number directly below $x$ in the grid, this lemma states that as we move downwards in a column, $\Delta$ stays the same, unless we reach a multiple of $a_n$, in which case $\Delta$ increases by $1$. Combined with Lemma \ref{primitive semigroup}, we now have the following complete pictorial description of the $\Delta$-function:

\begin{itemize}
    \item We give every multiple of $a_n$ a blue border. If a blue bordered box $sa_n$ has $s \in H$ and $s-\alpha \not\in H$, then we call that box \textit{primitive}. Lemma \ref{primitive semigroup} says the $\Delta$-function on primitive boxes evaluates to 1.
    \item Note there is exactly one primitive blue box in each column: $sa_n$ being in any particular column is equivalent to $s$ being a particular residue modulo $\alpha$. There is exactly one minimal element of $H$ within any residue class modulo $\alpha$.
    \item As we move downwards in a column, the value of $\Delta$ stays the same, unless we reach a multiple of $a_n$, in which case $\Delta$ increases by $1$.
\end{itemize}

Figure \ref{fig:3729 complete delta picture} is an image of the grid for $(a_1,a_2,a_3)=(3,7,29)$. The actual integers in each cell have been removed. The blue borders mark multiples of $a_3=29$, and the primitive blue boxes are those at the top of the light green columns.
\begin{figure}[h]
    \centering
    \includegraphics[width=3.5cm]{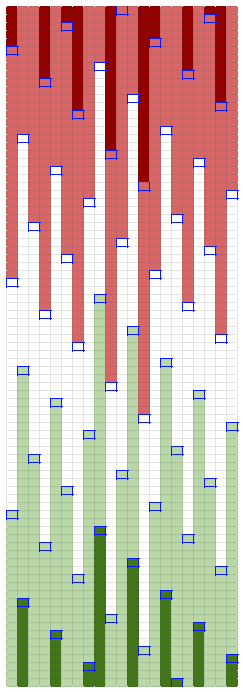}
    \caption{Zoomed-out image of $\Delta$-function for $\Sigma(3,7,29)$ without labels in boxes. White cells mean $\Delta=0$, light green means $\Delta=1$, dark green means $\Delta=2$, light red means $\Delta=-1$, and dark red means $\Delta=-2$.}
    \label{fig:3729 complete delta picture}
\end{figure}

\noindent Finally, we introduce the notation 
\[\chi(x):= \tau(x)-\tau(x-\alpha) = \Delta(x-\alpha)+ \cdots + \Delta(x-1)\]
for all $x \in \mathbb{Z}$.
\begin{lemma}\label{nondecreasing chi}
For all $x$, we have 
\[\chi(x+1)-\chi(x)= \begin{cases}
    1 &\text{if }x\text{ is a multiple of }a_n\\
    0 &\text{otherwise}
\end{cases}.\]
\end{lemma}

\begin{proof}
Note that $\chi(x+1)-\chi(x) = \Delta(x)-\Delta(x-\alpha)$. The result follows from Lemma \ref{delta in columns}.
\end{proof}

For the rest of this section, we work with a given Seifert homology sphere $Y=\Sigma(a_1, a_2, \dots , a_n)$, we define \[t=\frac{1}{2}\left((n-2)\alpha - \sum_{i=1}^{n-1} \frac{\alpha}{a_i}+1\right).\]
Also recall
\[N_0 =(n-2)a_1a_2 \cdots a_n -\sum_{i=0}^{n}\frac{a_1a_2 \cdots a_n}{a_i}= (2t-1)a_n - \alpha.\]
Denote $a_n'=a_n+\alpha$, and $Y'=\Sigma(a_1,a_2, \dots , a_n')$.

\begin{definition}
We say that a blue bordered box $sa_n$ is \textit{greening} if $\Delta(sa_n)>0$ (equivalently, $s \in H$). Otherwise, if $s \not\in H$ and $\Delta(sa_n) \le 0$, we say that the blue bordered box $sa_n$ is \textit{reddening}. This is because these boxes are where the values of the $\Delta$-function change color within the column.
\end{definition}

Notice that for any cell $x$, at least one of $\#(\text{greening borders nonstrictly above }x)$ and $\#(\text{reddening borders strictly below }x)$ is zero (where above and below here refer to cells in the same column as $x$), and that the former minus the latter is equal to $\Delta(x)$.

Next, we wish to prove an important result about $\chi(0)$, but to do this we must first show something important about the numerical semigroup generated by $\frac{a_1a_2\cdots a_n}{a_i}$ for a list of pairwise relatively prime positive integers $a_1,a_2,\dots,a_n$:

\begin{lemma}\label{largest nonelement}
For pairwise relatively prime positive integers $a_1,a_2,\dots,a_n$ where we define $P:=a_1a_2\cdots a_n$, the largest positive integral nonelement of the numerical semigroup $G$ minimally generated by $\frac{P}{a_i}$ for $i\in[1,n]$ is precisely
\[(n-1)P-\sum_{i=1}^{n}\frac{P}{a_i}.\]
\end{lemma}

\begin{proof}
As before, let $\alpha=\frac{P}{a_n}$ and consider the semigroup $H$ minimally generated by the set $\{\frac{\alpha}{a_i}\}_{i\in[1,n-1]}$, and let $a_nH$ be the set when all elements are multiplied by $a_n$. Consider arranging the natural numbers in an infinite table, with columns corresponding to residues modulo $\alpha$, similar to the grid in Section \ref{properties of tau}. Color elements of $H$ purple. Any term in the table on or below a purple term in the same column will be in $G$. This is because all elements of $G$ can be expressed as some element of $a_nH$ plus some nonnegative multiple of $\alpha$, which corresponds to shifting downwards in the same column. Thus, it suffices to find the largest purple term with no purple terms above it, and subtract $\alpha$ from it. 

Every purple term can be written in the form $\displaystyle a_n\sum_{i=1}^{n-1} \frac{c_i\alpha}{a_i}$ for some nonnegative integers $c_1,c_2,\ldots,c_{n-1}$. Now, note that two purple terms $\displaystyle a_n\sum_{i=1}^{n-1} \frac{c_i\alpha}{a_i}$ and $\displaystyle a_n \sum_{i=1}^{n-1} \frac{d_i\alpha}{a_i}$ are equal if and only if 
\[P\sum_{i=1}^{n-1} \frac{c_i}{a_i} \equiv P\sum_{i=1}^{n-1} \frac{d_i}{a_i} \pmod{\alpha}.\]

Since both sides are equivalent modulo $a_1$, we must have $\frac{c_1P}{a_1} \equiv \frac{d_1P}{a_1} \pmod{a_1}$, so we have $c_1 \equiv d_1 \pmod{a_1}$. Repeating this for all $i\in[1,n-1]$, we get $c_i \equiv d_i \pmod{a_i}$ for all $1 \leq i \leq n-1$. 

Therefore, in order to get the maximum purple term $\displaystyle a_n\sum_{i=1}^{n-1} \frac{c_i\alpha}{a_i}$ with no purple terms above it, we need to set set $c_i = a_i-1$ for all $1 \leq i \leq n-1$. The maximum purple term comes out to 
\[a_n\alpha(n-1)-a_n\sum_{i=1}^{n-1} \frac{\alpha}{a_i} = (n-1)P - \sum_{i=1}^{n-1} \frac{P}{a_i},\]
and subtracting $\alpha=\frac{P}{a_n}$ gives the desired result.
\end{proof}

\begin{remark}
When $n=2$, we recognize this as the well-known Chicken McNugget Theorem: the largest positive integer that cannot be expressed as a nonnegative integral linear combination of two relatively prime positive integers $p,q$ is indeed $pq-p-q$.
\end{remark}

\begin{lemma}\label{chi of zero}
We have that $\chi(0)=-t$.
\end{lemma}

\begin{proof}
Note that $\chi(0)=\Delta(-\alpha)+ \cdots + \Delta(-1)$. Since all negative numbers have nonpositive $\Delta$ values, $\chi(0)$ is equal to the total number of nonnegative reddening boxes, i.e. the number of nonnegative $s \not\in H$. By Lemma~\ref{largest nonelement}, the maximal nonelement of $H$ is 
\[2t-1=(n-2)\alpha-\sum_{i=1}^{n-1}\frac{\alpha}{a_i}.\]
Furthermore, out of the $2t$ blue bordered numbers $0,a_n, 2a_n, \cdots , (2t-1)a_n$, exactly half are reddening. This is because by Theorem \ref{deltadesc} (noting $N_0+\alpha=(2t-1)a_n$),
\[1=\Delta(sa_n)+\left(\Delta(N_0-sa_n)+1\right)=\Delta(sa_n)+\Delta((2t-1)a_n-sa_n),\]
so exactly one of $sa_n$ and $(2t-1)a_n-sa_n$ are reddening. Hence there are exactly $t$ nonnegative reddening boxes.
\end{proof}

Denote the $\Delta$-functions of $Y$ and $Y'$ as $\Delta(\bullet)$ and $\Delta'(\bullet)$ respectively. Consider the grids representing these two $\Delta$-functions. Observe that the following transformation on the $\Delta$ grid turns it into the $\Delta'$ grid:
\begin{itemize}
    \item Take each blue border, say around some number $sa_n$, and move it down $s$ cells to the number $sa_n+s\alpha=sa_n'$.
    \item In addition, change the $\Delta$ values by stipulating that the $\Delta$ values inside each blue border remains the same (note we have $\Delta(sa_n)=\Delta'(sa_n')$ since $sa_n$ is primitive if and only if $sa_n'$ is primitive). We then enforce Lemma \ref{delta in columns}, that the $\Delta$ values inside each cell is the same as the one above it, unless it has a blue border (in which case it is one greater). 
\end{itemize}
The following important observation about $\chi$ is clear from Lemmas \ref{nondecreasing chi} and \ref{chi of zero}.

\begin{lemma}\label{chi across picture}
We have that
\begin{enumerate}
    \item $\chi(x) < 0$ if $x \leq (t-1)a_n$,
    \item $\chi(x) = 0$ if $(t-1)a_n+1 < x \leq ta_n$,
    \item $\chi(x) > 0$ if $x > ta_n$. 
\end{enumerate}
\end{lemma}

This and the later Theorem \ref{glob min crit strip} motivates the following definition:

\begin{definition}
We call the interval $((t-1)a_n-\alpha,(t-1)a_n]$ the \textit{critical strip}. Note that the critical strip consists of the $\alpha$ consecutive numbers up to the $(t-1)$-th blue bordered box.
\end{definition}


The critical strip two very important properties, given in the following two theorems.

\begin{theorem}\label{glob min crit strip}
The first occurrence of the global minimum $\displaystyle\min_{x\ge 0}\tau(x)$ of the $\tau$-sequence occurs in the critical strip. That is, if $m$ is the smallest positive integer such that $\displaystyle\tau(m)=\min_{x\ge 0}\tau(x)$, then $m\in((t-1)a_n-\alpha,(t-1)a_n]$.
\end{theorem}

\begin{proof}
Note that by Lemma \ref{chi across picture}, if $x\le (t-1)a_n-\alpha$, then $\chi(x+\alpha)<0$, so $\tau(x+\alpha)<\tau(x)$. In addition, if $x>(t-1)a_n$, then $\chi(x)\ge 0$, so $\tau(x-\alpha)\le \tau(x)$. Therefore, the first occurrence of the global minimum of $\tau$ occurs in the critical strip.
\end{proof}

\section{Results on $d$-Invariants}\label{results d invariants}

In this section, we prove the following theorem. 

\begin{theorem}\label{three}
The $d$-invariants of the Seifert homology spheres $Y=\Sigma(a_1,a_2,\dots,a_n)$ and $Y'=\Sigma(a_1,a_2,\dots,a_n+\alpha)$ are equal for all pairwise relatively prime positive integers $a_1,a_2,\dots,a_n$.
\end{theorem}

In order to do this, we will use the following computational formula from \cite{borodzik2013heegaard}.

\begin{theorem}\label{d formula}
The $d$-invariant of a Seifert homology sphere can be computed as
\[d(Y)=d(\Sigma(a_1,a_2,\ldots,a_n)) = \frac{1}{4}\left(\varepsilon^2 e + e + 5 - 12\sum_{i=1}^n s(b_i, a_i)\right)-2\min_{x \geq 0} \tau_Y(x),\]
where, recall, $(e_0,b_1,b_2,\ldots,b_n)$ satisfy 
\begin{align}\label{eq0}
e_0P + P\sum_{i=1}^n \frac{b_i}{a_i} = -1
\end{align}
with $1\le b_i<a_i$ for all $i\in[1,n]$, where $P=\prod_{i=1}^na_i$. We also define 
\[e = e_0 + \sum_{i=1}^n \frac{b_i}{a_i} \text{ and } \varepsilon = \frac{1}{e}\left(-(n-2) + \sum_{i=1}^n \frac{1}{a_i}\right).\] 
Note that $\tau_Y$ is the $\tau$-sequence for $Y$ and that $s(b_i,a_i)$ is a Dedekind sum, given by the formula
\[s(h,k):=\sum_{i=1}^{k-1}\left\langle\frac{i}{k}\right\rangle\left\langle\frac{hi}{k}\right\rangle,\]
where $\langle \bullet\rangle$ is the sawtooth function
\[\langle x\rangle :=\begin{cases}
0 & x\in\mathbb{Z}\\
x-\lfloor x\rfloor-\frac{1}{2} & x\not\in\mathbb{Z}
\end{cases}.\]
Also note that the first term in the $d$-invariant formula is precisely the global grading shift $\frac{K^2+|\Gamma|}{4}$ for the canonical cohomology class $K$.
\end{theorem}

There are two main components to the formula for the $d$-invariant in Theorem \ref{d formula}: the first term, which is ultimately some function of $a_1,a_2, \dots, a_{n-1}$; and the second term, which depends on the $\tau$ sequence for $Y$. 

\subsection{Calculating the required difference in global minima of $\tau$-sequences}

To show that $d(Y')-d(Y)=0$, we start by computing the difference in the first terms to reduce it to a problem in finding the difference in the global minima of $\tau_Y$ and $\tau_{Y'}$.

\begin{lemma}\label{messy d formula}
To prove Theorem~\ref{three}, it is equivalent to demonstrate that
\[\min_{x \geq 0} \tau_{Y'}(x) - \min_{x \geq 0} \tau_{Y}(x) = -\frac{1}{8}\left(\left((n-2)\frac{P}{a_n} - \sum_{i=1}^{n-1} \frac{P}{a_na_i}\right)^2-1\right).\]
Note that the right hand side is only a function of $a_1,a_2,\dots,a_{n-1}$.
\end{lemma}

Dividing both sides of equation (\ref{eq0}) in the statement of Theorem~\ref{d formula} by $P$ gives $e=-\frac{1}{P}$, so $\varepsilon = (n-2)P - \sum_{i=1}^n \frac{P}{a_i}$. We now plug these into our equation for the $d$-invariant given by the theorem to get 
\[d(Y) = \frac{1}{4}\left(\underbrace{-\left((n-2)P - \sum_{i=1}^n \frac{P}{a_i}\right)^2\cdot \frac{1}{P}
}_{d_1(Y)}\underbrace{-\frac{1}{P} + 5 - 12\sum_{i=1}^n s(b_i,a_i)}_{d_2(Y)}\right) - 2\min_{x \geq 0} \tau_{Y}(x).\]
Define $d_1(Y)$ and $d_2(Y)$ as above, so that $\displaystyle d(Y)=\frac{1}{4}(-d_1(Y)+d_2(Y))-2\min_{x\ge 0}\tau_{Y}(x)$. Now, note that
\[d_1(Y)=-(n-2)^2P -\sum_{i=1}^n \frac{P}{a_i^2} - \sum_{1 \leq i < j \leq n} \frac{2P}{a_ia_j} + \sum_{i=1}^n \frac{2(n-2)P}{a_i}.\]
We will consider the $d_2(Y)$ and $\displaystyle -2\min_{x\ge 0}\tau_{Y}(x)$ terms later.
We now perform the same $d$-invariant computation for the Brieskorn sphere $Y'=\Sigma(a_1,a_2,\ldots,a_{n-1},a_n+\frac{P}{a_n})$. In what follows, let $a_n'=a_n+\frac{P}{a_n}$ and $P' = \frac{Pa_n'}{a_n}$. Similarly to the previous sphere,
\[d(Y') = \frac{1}{4}\left(\varepsilon'^2e' + e' + 5 - 12\sum_{i=1}^{n-1} s(b_i', a_i) - 12s(b_n',a_n')\right)-2\min_{x \geq 0} \tau_{Y'}(x),\]
where $(e_0',b_1',b_2',\ldots,b_{n-1}',b_n')$ satisfy
\[e_0'P'+P'\left(\sum_{i=1}^{n-1} \frac{b_i'}{a_i} + \frac{b_n'}{a_n'}\right)=-1,\]
with $b_i \in [1,a_i-1]$ for all $1 \leq i \leq n-1$ and $B \in [1,A-1]$. Also, as before, we have 
\[e' = e_0'+\sum_{i=1}^{n-1} \frac{b_i'}{a_i} + \frac{b_n'}{a_n'} \text{ and } \varepsilon' = \frac{1}{e'}\left(-(n-2)+\sum_{i=1}^{n-1} \frac{1}{a_i} + \frac{1}{a_n'}\right).\]
Thus, we have $e' = -\frac{1}{P'}$ and $\varepsilon'=(n-2)P' - \sum_{i=1}^{n-1} \frac{P'}{a_i} - \frac{P'}{a_n'}$. Substituting, we have 
\begin{align*}
d(Y') = \frac{1}{4}\left(\underbrace{-\left((n-2)P'-\sum_{i=1}^{n-1} \frac{P'}{a_i} - \frac{P'}{a_n'}\right)^2\frac{1}{P'}}_{d_1(Y')}\underbrace{- \frac{1}{P'} + 5 - 12\sum_{i=1}^{n-1} s(b_i',a_i) - 12s(b_n',a_n')}_{d_2(Y')}\right) \\
- 2\min_{x \geq 0} \tau_{Y'}(x).
\end{align*} 
Once again, we define $d_1(Y')$ and $d_2(Y')$ in such a way that $\displaystyle d(Y')=\frac{1}{4}(d_1(Y')+d_2(Y'))-2\min_{x\ge 0}\tau_{Y'}(x)$. Again, we expand the first term inside the parentheses to get 
\[d_1(Y')=-(n-2)^2P' - \sum_{i=1}^{n-1} \frac{P'}{a_i^2} - \frac{P'}{a_n'^2} - \sum_{1 \leq i < j \leq n-1} \frac{2P'}{a_ia_j} - \sum_{i=1}^{n-1} \frac{2P'}{a_n'a_i} + \sum_{i=1}^{n-1} \frac{2(n-2)P'}{a_i} + \frac{2(n-2)P'}{a_n'}.\]
We now compute $d_1(Y')-d_1(Y)$. We make heavy use of the identities $a_n'-a_n=\frac{P}{a_n}$, $P'-P = \frac{P^2}{a_n^2}$, and $\frac{P'}{a_n'} = \frac{P}{a_n}$ to eliminate all occurrences of $a_n'$ and $P'$ from this difference. All of these identities are easily verified from the definitions of $a_n'$ and $P'$. With some computation, the difference comes out to:

\[d_1(Y')-d_1(Y)=-(n-2)^2\cdot \frac{P^2}{a_n^2} - \sum_{i=1}^{n-1} \frac{P^2}{a_n^2a_i^2} + \frac{P^2}{a_n^2}\cdot\frac{1}{a_n'a_n} - \sum_{1 \leq i < j \leq n-1} \frac{2P^2}{a_n^2 a_i a_j} + \sum_{i=1}^{n-1} \frac{2(n-2)P^2}{a_n^2a_i}.\]
We can also evaluate the second difference term as

\begin{align*}
    d_2(Y')-d_2(Y)&=\frac{1}{P}-5+12\sum_{i=1}^ns(b_i,a_i)-\frac{1}{P'}+5-12\sum_{i=1}^{n-1}s(b_i',a_i)+12s(b_n',a_n')\\
    &=\frac{1}{a_n'a_n}+12s(b_n,a_n)-12s(b_n',a_n').
\end{align*}

The final simplification is due to the fact that $b_i'=b_i$ for all $1 \leq i \leq n-1$, which is true by equation (\ref{computing b}).

To evaluate this, we use the following facts about Dedekind sums from~\cite{shipp1965table}:
\begin{itemize}
    \item $s(h,k)=-s(-h,k)$.
    \item If $h_1\equiv h_2\pmod{k}$, then $s(h_1,k)=s(h_2,k)$.
    \item When $h_1h_2 \equiv 1 \pmod{k}$, $s(h_1,k) = s(h_2,k)$.
    \item $s(h,k) = -\frac{1}{4} + \frac{1}{12}\left(\frac{h}{k} + \frac{k}{h} + \frac{1}{kh}\right)-s(k,h)$ (The Dedekind Reciprocity Law).
\end{itemize}
Recall that $b_n \equiv -\frac{1}{a_1a_2\cdots a_{n-1}} \pmod{a_n}$. Similarly, $b_n' \equiv -\frac{1}{a_1a_2\cdots a_{n-1}} \pmod{a_n'}$. Thus,
\begin{align*}
    s(b_n,a_n) &= s(-a_1a_2\cdots a_{n-1},a_n)\\
    &= -s(a_1a_2\cdots a_{n-1},a_n) \\ 
    &=\frac{1}{4}-\frac{1}{12}\left(\frac{a_1a_2\cdots a_{n-1}}{a_n} + \frac{a_n}{a_1a_2 \cdots a_{n-1}} + \frac{1}{a_1a_2\cdots a_n}\right) + s(a_n,a_1a_2\ldots a_{n-1}),
\end{align*}
where we use the Dedekind Reciprocity Law. Similarly, \[s(b_n',a_n') = \frac{1}{4} - \frac{1}{12}\left(\frac{a_1a_2\ldots a_{n-1}}{a_n'} + \frac{a_n'}{a_1a_2 \ldots a_{n-1}} + \frac{1}{a_1a_2 \ldots a_{n-1}a_n'}\right) + s(a_n',a_1a_2\ldots a_{n-1}).\]
Subtracting, we get 
\begin{align*}
   12s(b_n,a_n)-12s(b_n',a_n') = \frac{a_1a_2\cdots a_{n-1}}{a_n'} + \frac{a_n'}{a_1a_2 \cdots a_{n-1}} + \frac{1}{a_1a_2\cdots a_{n-1}a_n'} \\ -\frac{a_1a_2\cdots a_{n-1}}{a_n} - \frac{a_n}{a_1a_2\cdots a_{n-1}} - \frac{1}{a_1a_2\cdots a_n}.
\end{align*}
Using the identity $a_n'-a_n=a_1a_2\cdots a_{n-1}$, we get 
\[12s(b_n,a_n)-12s(b_n',a_n') = \frac{-(a_1a_2\cdots a_{n-1})^2 + a_n'a_n - 1}{a_n'a_n}\]
after a bit of computation. Thus, we have:
\begin{align*}
    d(Y')-d(Y) = \frac{1}{4}\Bigg( -(n-2)^2\left(\frac{P^2}{a_n^2}\right) - \sum_{i=1}^{n-1} \frac{P^2}{a_n^2a_i^2} + \frac{P^2}{a_n^2}\left(\frac{1}{a_n'a_n}\right) -  \sum_{1 \leq i < j \leq n-1} \frac{2P^2}{a_n^2 a_i a_j}\\
    +\sum_{i=1}^{n-1} \frac{2(n-2)P^2}{a_n^2a_i} + \frac{1}{a_n'a_n} + \frac{-(a_1a_2\cdots a_{n-1})^2 + a_n'a_n - 1}{a_n'a_n}\Bigg) + 2\min_{x \geq 0} \tau_{Y}(x) - 2\min_{x \geq 0} \tau_{Y'}(x).
\end{align*}
For the $d$-invariants to be equal, we must have
\begin{align*}
    \min_{x \ge 0} \tau_{Y'}(x) - \min_{x \geq 0} \tau_{Y}(x) = \frac{1}{8}\Bigg(-(n-2)^2 \left(\frac{P^2}{a_n^2}\right) - \sum_{i=1}^{n-1} \frac{P^2}{a_n^2a_i^2} + \frac{P^2}{a_n^2}\left(\frac{1}{a_n'a_n}\right) \\
    -  \sum_{1 \leq i < j \leq n-1} \frac{2P^2}{a_n^2 a_i a_j}+\sum_{i=1}^{n-1} \frac{2(n-2)P^2}{a_n^2a_i} + \frac{1}{a_n'a_n} + \frac{-(a_1a_2\cdots a_{n-1})^2 + a_n'a_n - 1}{a_n'a_n}\Bigg).
\end{align*}
Note that $-(a_1a_2 \cdots a_{n-1})^2 = -\frac{P^2}{a_n^2}$, so the $\frac{P^2}{a_n^2}\left(\frac{1}{a_n'a_n}\right)$ and $-\frac{(-a_1a_2\cdots a_{n-1})^2}{a_n'a_n}$ cancel. Additionally, the $\frac{1}{a_n'a_n}$ and $-\frac{1}{a_n'a_n}$ terms cancel, so that we are left with 
\begin{align*}
    \min_{x \ge 0} \tau_{Y'}(x) - \min_{x \ge 0} \tau_{Y}(x) = \frac{1}{8}\Bigg(-(n-2)^2 \left(\frac{P^2}{a_n^2}\right) - \sum_{i=1}^{n-1} \frac{P^2}{a_n^2a_i^2} -  \sum_{1 \leq i < j \leq n-1} \frac{2P^2}{a_n^2 a_i a_j}\\
    +\sum_{i=1}^{n-1} \frac{2(n-2)P^2}{a_n^2a_i} + 1\Bigg).
\end{align*}
The expression inside the parentheses factors into 
\[1-\left((n-2)\frac{P}{a_n} - \sum_{i=1}^{n-1} \frac{P}{a_na_i}\right)^2,\]
which can be checked manually by expansion. This gives the desired result.\qed

\begin{remark}
Surprisingly, this quantity is tied to the numerical semigroup generated by the $n-1$ elements $\frac{\alpha}{a_i}$ for $i\in[0,n-1]$. Applying Lemma~\ref{largest nonelement} in the case where $n$ is replaced with $n-1$ gives that the maximal nonelement of the numerical semigroup $H$ is precisely
\[(n-2)\frac{P}{a_n}-\sum_{i=1}^{n-1}\frac{P}{a_n}{a_i},\]
which is the term found in Lemma \ref{messy d formula}.
\end{remark}

\subsection{Calculating the difference in global minima of $\tau$-sequences}

Using Lemma \ref{messy d formula} and our results about the $\Delta$-function and $\tau$-sequence in Section \ref{properties of tau}, we are now ready to prove Theorem \ref{three}.

\begin{proof}

We first note the following lemma:

\begin{lemma}
Under the grid transformation from $\Delta$ to $\Delta'$, the values of the $\Delta$- and $\Delta'$-functions in the critical strips of their respective grids remains the same. This is equivalent to the color schemes of the $\Delta$ and $\Delta'$ grids being the same in their respective critical strips.
\end{lemma}

\begin{proof}
Each blue border remains in its original column, and the relative ordering of the blue borders remains the same, so for any given cell in the critical strip, the quantity 
\[\#(\text{greening border nonstrictly above})-\#(\text{reddening border strictly below})\]
remains the same after the transformation from the $\Delta$ grid to the $\Delta'$ grid.
\end{proof}

In particular, this means that the first occurrence of the global minimum of the $\tau$-sequence remains in the same position relative to the critical strip, so
\[\min_{x\ge 0}\tau_{Y'}(x)-\min_{x\ge 0}\tau_{Y}(x)=\sum_{x=0}^{(t-1)a_n'}\Delta'(x) - \sum_{x=0}^{(t-1)a_n}\Delta(x),\]
since $(t-1)a_n'$ and $(t-1)a_n$ are the endpoints of the critical strips of $Y'$ and $Y$, respectively. Denote the interval $[0,(t-1)a_n]$ in the grid of $\Delta$ as $\mathcal{R}$, and similarly define $\mathcal{R}'$ for $\Delta'$ as the interval $[0,(t-1)a_n']$.

Note that $\displaystyle \sum_{x=0}^{(t-1)a_n}\Delta(x)$ equals
\begin{align*}
\sum_{x=0}^{(t-1)a_n}&\#(\text{greening borders nonstrictly above }x)-\#(\text{reddening borders strictly below }x)\\
    &=\sum_{\text{greening borders }g}\#(\text{cells nonstrictly below }g\text{ in }\mathcal{R})\\
    &-\sum_{\text{reddening borders }r}\#(\text{cells strictly above }r\text{ in }\mathcal{R}),
\end{align*}
with a similar statement holding for $\displaystyle\sum\limits_{x=0}^{(t-1)a_n'}\Delta'(x)$. Therefore, in order to compute the quantity 
\[\displaystyle\sum\limits_{x=0}^{(t-1)a_n'}\Delta'(x) - \displaystyle\sum\limits_{x=0}^{(t-1)a_n}\Delta(x),\] 
for each greening border $sa_n$ we consider the change in its ``contribution" to these sums before and after the transformation:
\begin{align*}
\delta(sa_n)&:=\#(\text{cells in }\mathcal{R}'\text{ nonstrictly below }sa_n')\\
&-\#(\text{cells in }\mathcal{R}\text{ nonstrictly below }sa_n).
\end{align*}
Similarly, for reddening borders $sa_n$ we define
\begin{align*}
\delta(sa_n)&:=\#(\text{cells in }\mathcal{R}'\text{ strictly above }sa_n')\\
&-\#(\text{cells in }\mathcal{R}\text{ strictly above }sa_n).
\end{align*}
Therefore,
\[\displaystyle\sum\limits_{x=0}^{(t-1)a_n'}\Delta'(x) - \displaystyle\sum\limits_{x=0}^{(t-1)a_n}\Delta(x)=\sum_{sa_n\text{ is greening}}\delta(sa_n)-\sum_{sa_n\text{ is reddening}}\delta(sa_n).\]
We compute these two sums separately:
\begin{itemize}
    \item Let there be $\ell$ greening borders in the $\mathcal{R}$ (hence there are $t-\ell$ reddening borders in the $\mathcal{R}$, since there are $t$ total blue borders in the $\mathcal{R}$). For each of these greening borders $sa_n$, under the transformation it moves down by $s$ cells, while the end of the critical strip moves down by $t-1$ cells. Hence $\delta(sa_n)=(t-1)-s$, so
    \[\sum_{sa_n\text{ is greening}}\delta(sa_n)=(t-1)\ell-\sum_{sa_n\text{ in }\mathcal{R}\text{ is greening}}s.\]
    \item As proven in Lemma~\ref{chi of zero}, there are exactly $t$ nonnegative reddening boxes, and we noted above that $t-\ell$ of them are in $\mathcal{R}$. For any reddening box below $\mathcal{R}$ (there are $\ell$ of these), its contribution changes under the transformation by $t-1$ due to the critical strip moving downwards by $t-1$ cells. For any reddening box $sa_n$ inside $\mathcal{R}$, it moves down by $s$ cells, and hence its contribution changes by $\delta(sa_n)=s$. So
    \[\sum_{sa_n\text{ is reddening}}\delta(sa_n)=(t-1)\ell + \sum_{sa_n\text{ in }\mathcal{R}\text{ is reddening}}s\]
\end{itemize}Subtracting these two quantities,  we obtain
\[\sum_{x=0}^{(t-1)a_n'}\Delta'(x) - \sum_{x=0}^{(t-1)a_n}\Delta(x) = -\sum_{\text{all }sa_n\text{ in }\mathcal{R}}s=-\frac{t(t-1)}{2}.\]
Substituting in our definition of $t$, we obtain exactly the condition that Lemma~\ref{messy d formula} requires in order to prove Theorem~\ref{three}.
\end{proof}

\section{Results on Maximal Monotone Subroots}\label{results maximal monotone subroots}



In this section, we define the \textit{maximal monotone subroot} of a graded root lattice homology, which was introduced recently in~\cite{dai2019involutive}. This maximal monotone subroot is a \textit{monotone graded root} which captures the general major structure of the graded root lattice homology.

\begin{figure}
    \centering
    \includegraphics[width=14cm]{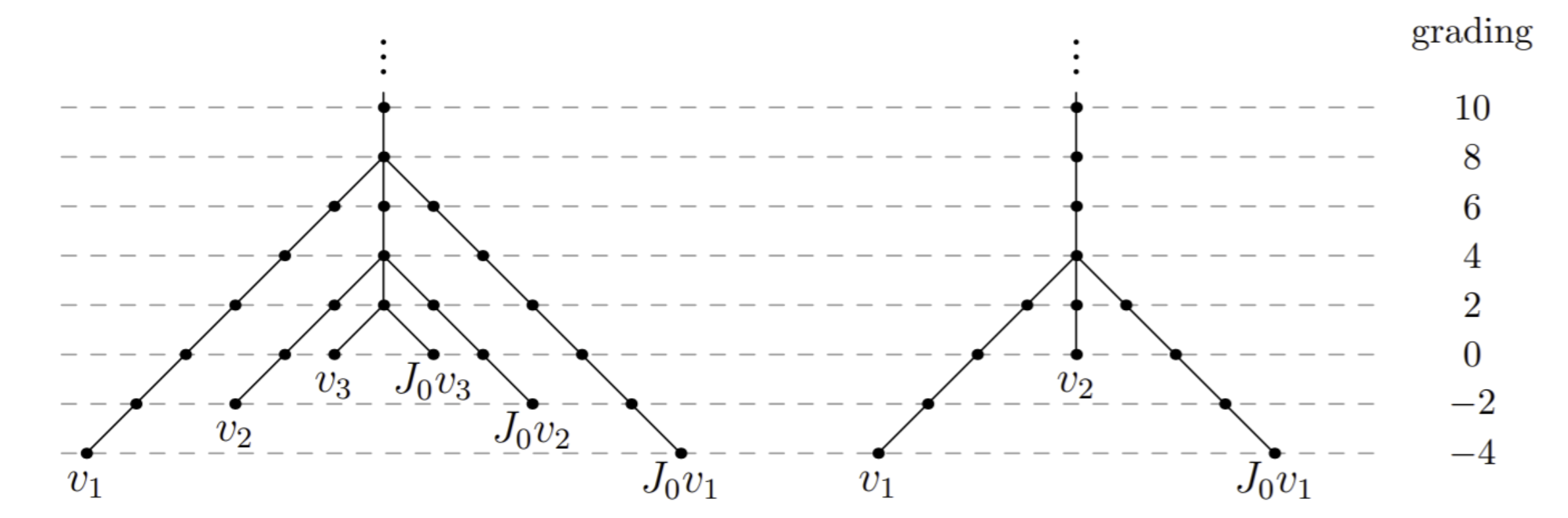}
    \caption{Two examples of monotone subroots. The left is the subroot $M(-4,8;-2,4;0,2)$ and the right is the subroot $M(-4,4;0,0)$.}
    \label{fig:monotone examples}
\end{figure}


First, we describe what a general monotone graded root is. 

\begin{definition}
Take a positive integer $n$ and two sequences of rational numbers $h_1<h_2<\cdots<h_n$ and $r_1>r_2>\cdots>r_n$ such that all the $h_i$ differ from each other by even integers, all the $r_i$ differ from each other by even integers, and $h_n \le r_n$. We construct the \textit{monotone graded root} $M=M(h_1,r_1;\ldots;h_n,r_n)$ as follows: 
\begin{enumerate}
    \item Form the stem by drawing an infinite tower upwards from a vertex at height $r_n$. This is said to be at \textit{grading} $r_n$. 
    \item For each $1 \leq i < n$, draw a pair of vertices $v_i$ and $J_0v_i$ at grading $h_i$, where $J_0v_i$ is the vertex $v_i$ reflected over the vertical axis of the graded root. Connect it to the stem at grading $r_i$.
\end{enumerate} 
(Note: if $h_n=r_n$, then $v_n=J_0v_n$ at grading $r_n$ in the second step). 
\end{definition}
Two examples are shown in Figure~\ref{fig:monotone examples}. 

To define the maximal monotone subroot of a lattice homology graded root, we first define some auxiliary definitions:

\begin{definition}
For any node $v$ in a graded root, we define the infinite tower stretching up from $v$ as $\gamma_v$. 
\end{definition}

\begin{definition}
For a given vertex $v$, the first point at which $\gamma_v$ meets the stem is the \textit{base} of $v$, denoted $b(v)$.
\end{definition}

\begin{definition}
The \textit{cluster} $C_b$ based at some grading $b$ is the set of all vertices with base $b$.
\end{definition}

\begin{definition}
The \textit{tip} of a cluster $C_b$ is the pair of vertices in $C_b$ with minimal grading. If there is more than one pair, any one can be arbitrarily selected, and if $|C_b|=1$, the tip is the singular vertex in $C_b$.
\end{definition}

\begin{definition}
Given a lattice homology graded root, we construct its \textit{maximal monotone subroot} as follows:
\begin{enumerate}
    \item Begin at the vertex on the stem of the graded root with the smallest grading. Call this grading $r$, and add the tips of $C_r$ to the subroot in the same fashion as they appear in the lattice homology graded root. 
    \item Move to the vertex on the stem with the next smallest grading (say vertex $s$), and add the tips of $C_s$ to the subroot in the same fashion as they appear in the lattice homology graded root if and only if the tips have strictly smaller grading than any tips previously added.
    \item Continue this process until all clusters considered are trivial (have size 1). This must happen since the number of clusters in any lattice homology graded root is finite.
\end{enumerate}
\end{definition}

An example of a graded root and its maximal monotone subroot is shown in Figure~\ref{fig:max subroot}.

\begin{figure}
    \centering
    \includegraphics[width=12cm]{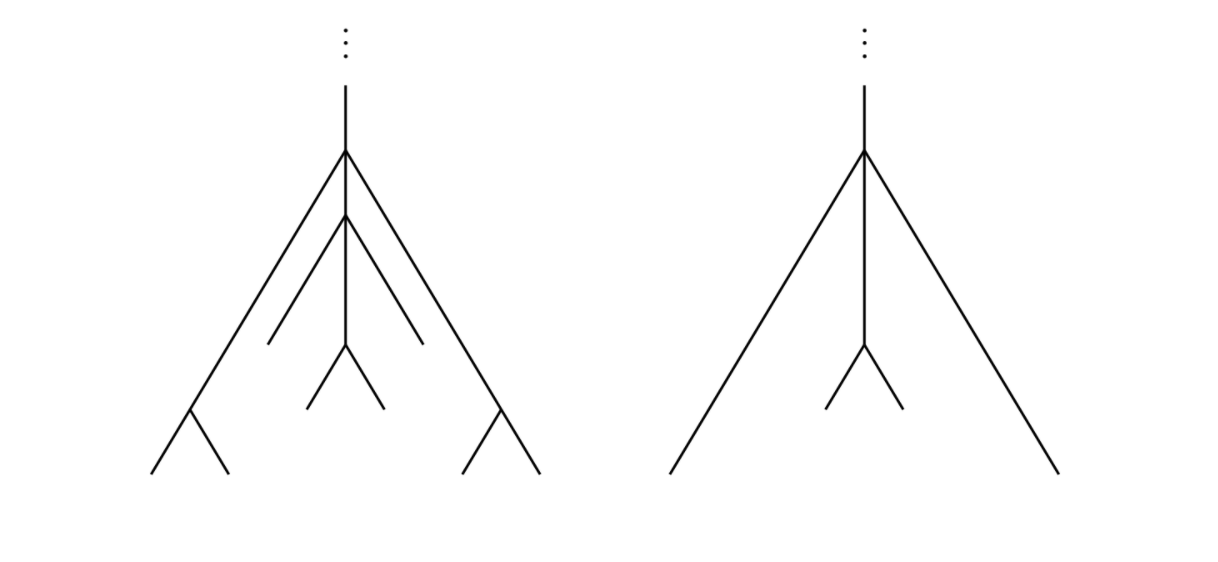}
    \caption{A graded root on the left and its maximal monotone subroot on the right.}
    \label{fig:max subroot}
\end{figure}

Maximal monotone subroots have high importance in relation to the graded roots of the lattice homologies of 3-manifolds, such as in the study of homology cobordism. First, we state the following helpful lemma, which is evident from the definition of the $\tau$-sequence and maximal monotone subroot.

\begin{lemma}\label{maximal monotone subroot global mins}
For any pair of symmetric global minima $\tau(m)$ and $\tau(N_0-m)$ of the $\tau$-sequence, the maximal monotone subroot is fully determined by the values of the $\Delta$-function in the interval $[m,N_0-m]$, up to a shift in grading.
\end{lemma}

As a corollary of the above, the values of the $\Delta$-function on any interval that contains two symmetric global minima of the $\tau$-seqeuence can fully determine the maximal monotone subroot up to a shift in grading.

\begin{theorem}\label{monotone root result}
The maximal monotone subroots of the lattice homologies of the Seifert homology spheres $Y=\Sigma(a_1,a_2,\dots,a_n)$ and $Y''=\Sigma(a_1,a_2,\dots,a_n+2\alpha)$ are the same.
\end{theorem}
\begin{proof}
Denote $a_n'':=a_n+2\alpha$. We will show that there is a pair of global minima $\tau_Y(m)$ and $\tau_Y(N_0-m)$ of $\tau_Y$ and a pair of global minima $\tau_{Y''}(m'')$ and $\tau_{Y''}(N_0''-m'')$ of $\tau_{Y''}$ such that the values of $\tau_Y$ on the interval $[m,N_0-m]$ and the values of $\tau_{Y''}$ on the interval $[m'',N_0''-m'']$ are identical, which finishes by Lemma~\ref{maximal monotone subroot global mins}, as these regions completely determine the maximal monotone subroots of the respective Seifert homology spheres. 

As shown in Theorem~\ref{glob min crit strip}, the first global minimum of $\tau_Y$ must occur in the critical strip $((t-1)a_n-\alpha,(t-1)a_n]$. Since the $\Delta$-function is antisymmetric under the map $x \mapsto N_0-x$ by statement 2 of Theorem~\ref{deltadesc}, we have that that $\tau$-sequence is symmetric under that map,. In particular, the last global minimum of $\tau_Y$ must occur in the region 
\[[N_0-(t-1)a_n,N_0-(t-1)a_n+\alpha)=[ta_n-\alpha, ta_n),\]
which is the image of the critical strip under this map. Now, note that the region $((t-1)a_n-\alpha, ta_n)$ contains two symmetric global minima. Furthermore, this interval corresponds to a centrally symmetric region of the graded root of $Y$ since the endpoints $(t-1)a_n-\alpha$ and $ta_n$ sum to $N_0=(2t-1)a_n-\alpha$. By Lemma~\ref{maximal monotone subroot global mins}, we can fully determine the monotone subroot (up to a shift in the grading) solely based on the $\Delta$ values in that region. 

The corresponding region in $Y''$ is $((t-1)a_n''-\alpha,ta_n'')$, but the length of this interval is not the same as the length of $((t-1)a_n-\alpha, ta_n)$. Instead, we will consider the interval
\[((t-1)a_n'', ta_n''-\alpha),\]
which is a centrally symmetric region of the graded root of $Y''$, and that the length of this interval is the same as that of $((t-1)a_n-\alpha,ta_n)$. 

By Lemma~\ref{maximal monotone subroot global mins}, it suffices to show that the sequence of $\Delta$ values within the interval $((t-1)a_n-\alpha,ta_n)$ are exactly the same as the $\Delta''$ values within the interval $((t-1)a_n'',ta_n''-\alpha)$, since these sequences fully determine the maximal monotone subroots of the lattice homologies $Y$ and $Y''$, respectively. To this end, note that for any $0< i < a_n+\alpha$, we have
\[\Delta((t-1)a_n-\alpha+i)=\Delta''((t-1)a_n''-\alpha+i)=\Delta''((t-1)a_n''+i)\]
where the first equality holds due to our discussion in Section~\ref{results d invariants} about the transformation from $\Delta$ to $\Delta'$ (and then to $\Delta''$), and the second equality follows from Lemma~\ref{delta in columns} since $(t-1)a_n''+i$ is never a multiple of $a_n''$. Thus, the maximal monotone subroots of $Y$ and $Y''$ are the same up to a shift in grading.

In addition, Theorem~\ref{three} guarantees that the $d$-invariants are the same, so the grading of the global minima of $Y$ and $Y''$ are the same. This means that, in fact, the maximal monotne subroots of $Y$ and $Y''$ are the same, as desired.
\end{proof}

\bibliographystyle{plain}
\bibliography{ref}

\end{document}